\documentclass[a4paper,UKenglish,cleveref,autoref,thm-restate]{lipics-v2021}
\usepackage{amsmath,amssymb,amsthm}
\usepackage{mathtools}
\usepackage{thmtools}
\usepackage{thm-restate}
\usepackage{hyperref}
\usepackage{cleveref}
\usepackage[ruled]{algorithm}
\usepackage[noend]{algpseudocode}
\usepackage{enumitem}
\usepackage{amsfonts}
\usepackage{amsmath} 
\usepackage{amssymb}
\usepackage{enumerate}

\newcommand{\PROOF}{\begin{proof}}

\newcommand{\QED}{\end{proof}}

\newcommand{\floor}[1]{ \left\lfloor #1 \right\rfloor }

\newcommand{\diam}{\mathrm{diam}}

\newcommand{\N}{\mathbb{N}}

\newcommand{\Q}{\mathbb{Q}}
\newcommand{\R}{\mathbb{R}}

\renewcommand{\dim}{{\mathrm{dim}}}

\newcommand{\Dim}{{\mathrm{Dim}}}

\newcommand\catenate{\mathbin{\text{\ttfamily\upshape ++}}}

\title{On the Constructive Dimension Spectrum of Polynomials} 

\titlerunning{On the Constructive Dimension Spectrum of Polynomials} 

 \author{Prajval Koul}{Department of Computer Science and Engineering, Indian Institute of Technology Kanpur, India. \and \url{https://prajvalkoul.github.io/}}{prajvalk21@iitk.ac.in}{https://orcid.org/0000-0002-8879-8330}{}

 \author{Satyadev Nandakumar}{Department of Computer Science and Engineering, Indian Institute of Technology Kanpur, Kanpur, Uttar Pradesh, India. \and \url{https://www.cse.iitk.ac.in/users/satyadev/} }{satyadev@cse.iitk.ac.in}{https://orcid.org/0000-0002-0214-0598}{}

 \authorrunning{P. Koul and S. Nandakumar} 

 \Copyright{Prajval Koul and Satyadev Nandakumar} 



\category{} 

\relatedversion{} 

\nolinenumbers 

\newcommand{\snt}[1]{\tilde{#1}}

\EventEditors{Sayan Bhattacharya, Danupon Nanongkai, Michael Benedikt, and Gabriele Puppis}
\EventNoEds{4}
\EventLongTitle{53rd International Colloquium on Automata, Languages, and Programming (ICALP 2026)}
\EventShortTitle{ICALP 2026}
\EventAcronym{ICALP}
\EventYear{2026}
\EventDate{July 7--10, 2026}
\EventLocation{Royal Holloway, University of London, Egham, United Kingdom}
\EventLogo{}
\SeriesVolume{374}
\ArticleNo{180}

\begin{document}
\maketitle

\begin{abstract}
Recently, Stull \cite{Stull2025}, \cite{Stull2022} resolved a long-standing open problem posed by Lutz, on whether the set of effective Hausdorff dimensions of points on a straight line in $\R^2$ - the effective dimension spectrum of the line - contains a unit interval. This question is related to problems in classical fractal geometry like the Kakeya conjecture and Furstenberg sets. Stull posed an open question on the dimension spectra of polynomial curves. 
  
For the first result, with new techniques which adapt the theory of classical real root-finding of polynomials to the current setting, we show that the dimension spectra of every polynomial curve contains at least two points. This answers an open question posed by Stull \cite{Stull2025}, \cite{Stull2022}. We use the main result to construct a class of polynomials which have width strictly greater than 1, answering a second problem stated in \cite{Stull2025},\cite{Stull2022}. 
  
Stull \cite{Stull2025} resolved the dimension spectrum conjecture for planar lines, showing that it contains a unit interval. For the second result, we resolve the conjecture for a subfamily of polynomials whose coefficients form a "low" dimension point in $\mathbb{R}^{d+1}$. 
\end{abstract}


\section{Introduction}
The theory of effective Hausdorff dimension was introduced by J. Lutz in 2000 \cite{Lutz2000b}, \cite{Lutz2003b} initially over the Cantor Space of infinite binary sequences as a tool to study the relationships between complexity classes \cite{Lutz2003}. Subsequent works have adapted the theory to study the effective Hausdorff dimensions and points in Euclidean space and in general metric spaces (\cite{Lutz2008a}, \cite{Lutz2008c}, \cite{Turetsky2011}, \cite{Lutz2018}, \cite{Lutz2018c}, \cite{Stull2020}, \cite{Stull2022}, \cite{Lutz2023a}, \cite{Mayordomo2018}). A major highlight of this theory is the ``point-to-set principle'', allowing classical Hausdorff dimensions to be derived using effective pointwise arguments. A central question arising in this setting concerns the effective dimension \emph{spectrum} of sets in the Euclidean plane: what values are attained by the effective dimensions of points inside sets in $\R^2$? 

In a recent work, Stull \cite{Stull2022}, extending an earlier work by N. Lutz and Stull \cite{Lutz2020e}, settled a long-standing conjecture of J. Lutz showing that the dimension spectrum of \emph{every} line in $\R^2$ contains a unit interval. Stull \cite{Stull2022} proposes extending the study to the dimensions of points along polynomial curves. 

In this work, we resolve Stull's problem for every univariate
polynomial with real coefficients. We show that its dimension spectrum of every polynomial contains at least two points. Further, we show that the dimension spectrum of every polynomial contains unit-length interval when the coefficients have low dimension. Our first main result is the following.
\begin{restatable}[]{thm}{maintheorem}\label{thm:d_degree_bound}
  For $(a_0,\dots,a_d)\in\R^{d+1}$, $x\in\R$, we have
  \begin{align}
    \label{ineq:degree_d}
    \dim\left(x,\sum_{i=0}^d a_i x^i\right)
    \ge
    \dim\left(x\mid a_0,\dots,a_d\right)+
    \min\left\{\dim\left(a_0,\dots,a_d\right),\dim^{a_0,\dots,a_d}(x)\right\}.
  \end{align}
\end{restatable}
The second main theorem for this paper establishes that the dimension spectrum of polynomials with coefficients with effective dimension at most 1, contains a unit-length interval. 
\begin{restatable}[]{thm}{maintheoremlow}\label{thm:less}
  Let $\sum_{i=0}^{d}a_ix^i\in\mathbb{R}[x]$ be a degree $d$
  polynomial with $\dim(\mathbf{a})\leq1$. Then for every $s\in[0,1]$,
  there is a point $x\in\mathbb{R}$ such that
  $\dim\left(x,\sum_{i=0}^{d}a_ix^i\right)=s+\dim(\mathbf{a})$.
\end{restatable}
The proofs adapt the techniques in N. Lutz and Stull \cite{Lutz2020e} to polynomials. We adapt the classical root-finding methods, namely, the bisection method, together with methods which count real roots from Sturm's theory \cite{Sturm1835} (see for example, von zur Gathen and Gerhard \cite{Gathen2013} and Yap \cite{Yap2000}) to establish our results. Due to the fact that the coefficients can be arbitrary real numbers, we have to adapt these techniques to form short descriptions of the roots of these functions. These versions may be of broader interest. 

The proof of the second result, building on the first, involves
encoding the coefficients of a given polynomial in such a way so as to control the dimension of the corresponding point. 

We conclude by discussing the implications of our result - in
particular, that dimension level sets - sets consisting of points of the same effective dimension - cannot contain polynomials. 


\section{Prerequisites}
We denote the binary alphabet by $\Sigma=\{0,1\}$. The set of finite binary strings is denoted by $\Sigma^*$ and the set of infinite binary sequences, by $\Sigma^\infty$. Empty string is denoted by $\lambda$. The length of a finite string $w\in\Sigma^*$ is denoted by $\ell(w)$. 

We denote the set of rational numbers by $\Q$ and the set of reals by $\R$. In this work, we assume a binary encoding $e:\Q \to \Sigma^*$ of the set of rationals. 

We denote by $\mathbf{a}$ the tuple $a_0,a_1,\ldots,a_d$, denoting the coefficients of a polynomial $\sum_{i=0}^da_ix^i\in\mathbb{R}[x]$, such that $a_d\ne0$. 

We now introduce the basic notions in Kolmogorov complexity (see, for example, Downey and Hirschfeldt \cite{Downey2008}, Nies \cite{Nies2009} or Li and Vit\'{a}nyi \cite{Li2019}). 

\begin{definition}[Kolmogorov Complexity of binary strings]
	For each pair of strings $v,w\in\{0,1\}^*$, the \emph{Kolmogorov
		complexity} of $v$ given $w$ is defined as $K\left(v|w\right)$ =
	$\min_{\pi\in\{0,1\}^*}$
	$ \left\{\ell(\pi):U\left(\pi,w\right)=v\right\}$, where $U$ is a
	fixed universal (prefix) Turing machine. The \emph{Kolmogorov
		complexity} of $v$, denoted by $K(v)$, is $K(v\mid \lambda)$.
\end{definition}
\begin{definition}[Kolmogorov Complexity of binary strings relative to an oracle] 
	For each pair of strings $v,w\in\{0,1\}^*$, the \emph{Kolmogorov complexity} of $v$ given $w$ relative to an oracle $A\subseteq\mathbb{N}$ is defined as $K^A\left(v|w\right)$=$\min_{\pi\in\{0,1\}^*}\left\{\ell(\pi):U^A\left(\pi,w\right)=v\right\}$, where $U$ is a fixed universal (prefix) Turing machine with oracle access to $A$. 
\end{definition}
In recent works (\cite{Lutz2008a}, \cite{Lutz2008c}, \cite{Turetsky2011}, \cite{Lutz2018}, \cite{Lutz2018c}, \cite{Stull2020}, \cite{Stull2022}, \cite{Lutz2023a}) the theory of Kolmogorov complexity of binary strings has been adapted to the study of Kolmogorov complexity of reals. Instead of defining the complexity in terms of the truncated binary expansions of the real (which has known issues such as addition of reals being uncomputable - see Weihrauch \cite{Weihrauch2000} for a discussion), the approach defines the complexity of a real in terms of the complexities of its rational approximations. The intuition behind the following definition is this: any rational $q$ within the open neighborhood of radius $2^{-r}$ around $x$ is a valid description of $x$ to within precision $r\in\N$. The shortest description of a real point $x$ in $\R^n$ to within a precision $r\in\mathbb{N}$ is the shortest description of any rational point within the $2^{-r}$ neighborhood of $x$. Note that this may not necessarily be the rational obtained by truncating the binary expansion of $x$ to $r$ bits. 

\begin{definition}[Kolmogorov Complexity of Reals](Lutz and Mayordomo
  \cite{Lutz2008a}) 
  The Kolmogorov complexity of a real number $x\in\mathbb{R}^n$ up to
  a precision $r\in\mathbb{N}$ is defined as
\begin{align*}
K_r(x)=\min\left\{K(q):q\in\mathbb{Q}^n\cap B\left(x,2^{-r}\right)\right\},
\end{align*}
where $B(x,d)=\left\{y\in\mathbb{R}^n:\|y-x\|_2< d\right\}$ is the open ball of radius $d$ centered at $x$.
\end{definition}
The conditional Kolmogorov complexity of $x\in\R^m$ given $y\in\R^n$
is defined using the Kolmogorov complexity of rational approximations
to $x$ and $y$. 

\begin{definition}[Conditional Kolmogorov Complexity of Reals]
  (Lutz and Mayordomo \cite{Lutz2008a}) 
  The conditional Kolmogorov complexity of $x\in\R^m$ at precision
  $r$ given $q\in\Q^n$ is
\begin{align*}
\hat{K}_r\left(x|q\right)=\min\left\{K\left(p|q\right):p\in B\left(x,2^{-r}\right)\cap\Q^m\right\}.
\end{align*}
The conditional Kolmogorov complexity of $x\in\mathbb{R}^m$ at
precision $r\in\mathbb{N}$ given $y \in \mathbb{R}^n$ at precision
$s\in\N$ is defined as
\begin{align*}
K_{r,s}\left(x|y\right)=\max\left\{\hat{K}_r(x|q):q\in B\left(y,2^{-s}\right)\cap \Q^n\right\}.
\end{align*}
\end{definition}

{\bf Notation.} We use $K_r\left(x|y\right)$ to denote $K_{r,r}\left(x|y\right)$. 

One of the characteristics of the theory of effective Hausdorff and packing dimension, in contrast to the classical theory, is that individual points can have strictly positive effective dimension. The effective Hausdorff dimension of a point is defined as follows. 

\begin{definition}[Effective Hausdorff dimension of a point](Lutz and Mayordomo \cite{Lutz2008a}) The \emph{effective Hausdorff dimension}
  of a point $x\in\R^m$, denoted $\dim(x)$, is defined by
  $ \dim(x)=\liminf_{r\rightarrow\infty}\frac{K_r(x)}{r}$. The
  \emph{effective strong dimension} of $x\in\mathbb{R}^m$ is defined
  by
$\Dim(x)=\limsup_{r\rightarrow\infty}\frac{K_r(x)}{r}$.
\end{definition}
Relativizing the above definitions with respect to an oracle $A \subseteq \N$, we can define the notions $K^A_r(x)$, $K^A_r(x|y)$, $\dim^A(x)$, and $\Dim^A(x)$, where the universal machine $U$ has access to the oracle $A$. Since there is a bijective correspondence between subsets of natural numbers and reals, using any standard encoding of reals, we may also define, for any $y\in\mathbb{R}$, the notions $K^y_r(x)$, $\dim^y(x)$, and $\Dim^y(x)$. 

\begin{definition}[Dimension spectrum of a set](Lutz and
Mayordomo\cite{Lutz2020d})
	For any $S \subseteq \R^n$, the \emph{dimension spectrum} of $S$, denoted by $\text{sp}(S)$, is defined by $\text{sp}(S)=\left\{\dim(x):x\in S\right\}$. 
\end{definition}
The following result by Lutz and Stull \cite{Lutz2020e} gives a lower bound on the dimensions of points in the graph of a straight line $z\mapsto az+b$ in terms of the dimensions of $a$, $b$ and the effective relative dimension of the co-ordinate $z\in\R$.

\begin{theorem}[\label{stullre}Lutz and Stull \cite{Lutz2020e}(unrelativized)] 
For every $a,b,x\in\mathbb{R}$, we have
\begin{align}\label{thm:line_dim_ub}
\dim\left(x,ax+b\right)\geq\dim^{a,b}\left(x\right)+\min\left\{dim(a,b),dim^{a,b}(x)\right\}.
\end{align}
\end{theorem}

\section{Basic properties of Kolmogorov complexity of reals}
In this section, we state a few basic properties of Kolmogorov complexity, conditional Kolmogorov complexity and relative Kolmogorov complexity of reals pertinent for the remainder of the work.

The following approximate symmetry of information holds for pairs of reals. 

\begin{lemma}[Lutz and Stull \cite{Stull2020}]
\label{lem:soi}
For every $m,n\in\N$, $x\in\R^m$, $y\in\R^n$, and precision parameters
$r,s\in\N$ with $s\le r$, we have
\begin{enumerate}
\item $\left|K_{r,r}(x|y)+K_r(y)-K_{r,r}(x,y)\right|~\le 
O(\log r)+O(\log\log\|y\|)$.
\item $|K_{r,s}(x|x)+K_s(x)-K_r(x)|~\le ~
O(\log r)+O(\log\log\|x\|)$.
\end{enumerate}
\end{lemma}
The following lemma establishes a one-sided bound between the
relativized Kolmogorov complexity and the conditional complexity, at a
specific precision.

\begin{lemma}[J. Lutz and N. Lutz \cite{Lutz2018}]
\label{lem:relative_conditional_complexity}
For all $m,n\in\N$, there is a constant $c\in\N$ such that for all
reals $x\in\R^m$, $y\in\R^n$, and precision parameters $r,t\in\N$ we
have
\begin{enumerate}
\item $K_r^y(x)\le K_{r,t}(x|y)+K(t)+c$.
\item $\dim^y(x)\le\dim(x|y)$.
\end{enumerate}
\end{lemma}
When the precision of either the conditioning variable or the
conditioned variable changes, then the Kolmogorov complexity changes
at most by an amount linear in the change in precision, uniformly in
the points, as the following lemmas show.

\begin{lemma}[Case and J. Lutz \cite{Case2015}]
\label{lem:precision_linear}
There is a constant $c\in\N$ such that for all $n\in\mathbb{N}$, $x\in\R^n$ and
precision parameters $r,s\in\N$, we have
\begin{align}
\label{ineq:precision_linear}
K_r(x)\ \le\ K_{r+s}(x) \ \le\ K_r(x)+K(r)+ns+a_s+c.
\end{align}
\end{lemma}
The following lemma states similar bounds for conditional Kolmogorov
complexity.

\begin{lemma}[J. Lutz and N. Lutz \cite{Lutz2018}]
  \label{lem:precision_var}
  For all $m,n\in\N$, there is a constant $c\in\N$ such that for all
  points $x\in\R^m$, $y\in\R^n$, $q\in\Q^n$ and precision parameters
  $r,s,t\in\N$, we have the following.
\begin{enumerate}
\item $\hat{K}_{r}(x|q)
  \ \le\ \hat{K}_{r+s}(x|q)
  \ \le\ \hat{K}_{r}(x|q)+ms+2\log(1+s)+K(r,s)+c$.
\item
  $K_{r,t}(x|y)
  \ \ge\ K_{r,t+s}(x|y)
  \ \ge\ K_{r,t}(x|y)-ns-2\log(1+s)+K(t,s)+c$.
\end{enumerate}
\end{lemma}

\section{Outline of the Proof of \cref{thm:d_degree_bound} }

Let $A(x)=\sum_{i=0}^d a_i x^i$. It is easy to establish an upper
bound on the dimension of a point $(x,A(x))$, given approximations to the coefficients of $A$ and to $x$. We can give a lower bound on the dimension of the vector $(x,a_0,\dots,a_d)$ using symmetry of information. However, the Kolmogorov complexity of the point $(x,A(x))$ may be lower than that of $(x,a_0,\dots,a_d)$. The main task we accomplish is to establish lower bounds for the complexity of $(x,A(x))$ and consequently, its effective Hausdorff dimension. 

The outline of the proof, adapting that of Lutz and Stull
\cite{Stull2020} is as follows. In \cref{lem:errbound}, we show that if a different $d$-degree polynomial $B(x)=\sum_{i=0}^d b_i x^i$ intersects $A$ at $x$, then either the coefficients of $B$ are close in $L_2$ norm to $A$, or the coefficients of $B$ must have very high Kolmogorov complexity. We introduce a technique of approximating roots when only approximations to the coefficients and rational approximations to points in the domain are available, adapting Sturm's theory to this new setting. This lower bounds the Kolmogorov complexity of $\left(x,A(x)\right)$. 

\cref{lem:oraclebound} (lemma 3.3 from \cite{Lutz2020e}) ensures an oracle $D$ which limits the complexity of the coefficients of $A$. These results are then used to obtain the main technical lemma, \cref{lem:errorestimate}, which yields the exact lower bound for the Kolmogorov complexity of $(x,A(x))$ at a specified precision $r$. The main lemma is then utilized to obtain a lower bound for the effective Hausdorff dimension of $\left(x,A(x)\right)$.

\section{Approximating intersecting polynomials}
\label{ssec:roots}


In this subsection, we prove a lower bound for the Kolmogorov complexity of \emph{any} approximation to the coefficients of a polynomial in terms of Kolmogorov complexity of the exact coefficients and the roots. This is crucial bound which leads to the main result (\cref{thm:d_degree_bound}). To this end, we provide an algorithm which lists approximations to all points where two given polynomials $B(x) = \sum_{i=0}^d b_i x^i$ and $A(x)=\sum_{i=0}^d a_i x^i$ intersect. This ensures that the complexity of any $x$ such that $A(x)=B(x)$ can be lower bounded using the complexities of ${\bf a}$ and ${\bf b}$. Since the coefficients and $x$ are real, this algorithm will take as input approximations to these values, and output approximations to the roots to any arbitrary precision. 

The points of intersection of $A$ and $B$ are exactly the roots of the $d$-degree polynomial $C(x) = A(x)-B(x)$. Thus there are at most $d$ points of intersection between $A$ and $B$ on $\R$. The problem at hand therefore reduces to computing all possible real roots, with multiplicities, of a polynomial with real coefficients, up to arbitrary precision. We start by outlining the bisection method, followed by the importance of Sturm's theorem, and then specify the algorithm. 



In the special case of polynomials, the following method enumerates \emph{all} the real roots of any univariate polynomial. Assume that we know the exact values of all the coefficients of a polynomial (in our case, we will have to work with approximations, and this makes the procedure much more technical). We start by getting a bound on the absolute value of any real root. This helps us to select an interval that is large enough to contain all the real roots of the polynomial. Recall that the roots of the polynomial are not guaranteed to be distinct - roots may have multiplicity greater than 1. In order to get all the real roots (with multiplicities), we use Sturm's theorem to get the number of distinct roots (with arbitrary multiplicities) of the polynomial in any given interval. This allows us to improve the bisection method whenever there are repeated roots (of even multiplicities) in any interval. 

For a degree $n$ polynomial $P$, we compute a sequence of polynomials, called the Sturm sequence of $P$ as follows (see, for example, Gerhard, von zur Gathen \cite[Ch~2]{Gathen2013}, or Yap \cite[Ch~7]{Yap2000}. The first 2 polynomials in the sequence are the polynomial itself, followed by its derivative - \emph{i.e} $P_0 = P$ and $P_1=P'$. For $i\ge2$, $P_i=P_{i-2}\mod P_{i-1}$, where $\mod$ is the remainder obtained when dividing the polynomial $P_{i-2}$ by $P_{i-1}$. Any Sturm sequence has at most $d+1$ elements, for a degree $d$ polynomial. Now, all the polynomials in the Sturm sequence obtained are evaluated at the end points of an interval, say $(a,b]$, which results in 2 sequences, each corresponding to an end point. We count the number of sign changes in each sequence, and denote them by $\sigma(a)$ and $\sigma(b)$ respectively. 

\begin{theorem}[Sturm, 1835 \cite{Sturm1835}]\label{thm:sturm}
	For a square free polynomial $P$, the number of distinct roots of $P$ in the interval $(a,b]$ is $\sigma(a)-\sigma(b)$. If the polynomial has repeated roots, and if neither $a$ nor $b$ is a root of $P$, then the number of distinct roots in $(a,b]$ is equal to $\sigma(a)-\sigma(b)$. 
\end{theorem}

If the end points happen to be the roots of the polynomial, we report them as is, and continue the search if required.

The following lemma bounds the roots of the polynomial $P$. We use this to determine the starting interval for our algorithm, ensuring that it is large enough to output all the real roots of $P$. 

\begin{lemma}[Cauchy Bound \cite{Cauchy1829}]
	Let $P(x)=\sum_{i=1}^d c_ix^i$ be a polynomial. If for any $x \in \R$, $P(x)=0$, then $x\in(-\beta,\beta)$, where $\beta=1+\max \left\{\left|\frac{a_{n-1}}{a_n}\right|,\dots,\left|\frac{a_0}{a_n}\right|\right\}$. 
\end{lemma}

Using the above classical results, we now introduce the algorithm that upper bounds the number of roots in an interval from rational approximations of the coefficients. We introduce the preliminary algorithms for the Sturm sequence and the approximate sign counting, concluding with the algorithm that enumerates the roots. 

The main subtlety we deal with is this: since the algorithms work with approximations, the exact signs of the polynomial value cannot be determined. This is because a small negative value is considered an acceptable input approximation to a small positive value. This leads to a conservative estimate of the sign changes, leading to a slightly longer list of possible roots. No root will be omitted at the given precision. However, since it is impossible to algorithmically determine whether a function is exactly equal to 0 or is a small positive or small negative value, points besides the roots may be counted if the precision is not sufficiently high. The algorithm outputs a list sufficiently short so that the Kolmogorov complexity of describing any root from the output list remains acceptably small. 

\begin{lemma}\label{lem:strumseqeval}
  There is an algorithm SturmSequence that, on input
  $(c_0, \dots, c_d)$ and $\alpha\in\Q$ outputs a list of rationals
  $\langle \hat{\alpha}_i \rangle_{i=1}^d$ of the evaluations of the
  Sturm Sequence corresponding to the polynomial
  $P_0(x)=\sum_{i=0}^d c_i x^i$ at $\alpha$.
\end{lemma}

The main step in our algorithm to find approximations to all real
roots is the following sign computation. If the values of the
polynomials are accurately known, then the sign determination is
trivial. However, when the true value is nearly 0, it is difficult to
determine its actual sign. We adopt a conservative upper bound for the
actual number of sign changes.

\begin{lemma}\label{lem:signchangecalc}
  There are algorithms \textsc{MaxSignChange} and
  \textsc{MinSignChange} for any sequence $(q_0, \dots, q_d)$ of
  rational approximations to reals $(r_0, \dots, r_d)$ where for $0
  \le i \le d$ we have $|q_i - r_i| < 2^{-r}$, \textsc{MaxSignChange}
  outputs an upper bound on the number of sign changes in $(r_0,
  \dots, r_d)$ and \textsc{MinSignChange} outputs a lower bound.
\end{lemma}
Thus, we get the following upper bound on the output list of possible
roots. This list is guaranteed to contain all the real roots of the
polynomial to the given precision. However, since
\cref{lem:signchangecalc} is a conservative upper bound on the number
of sign changes, there could be some values in the output list which
are not the roots of the polynomial. This cannot be avoided in
general. However, the output list of values is sufficiently small to
control the number of bits used to describe \emph{any} of its members.
This upper bounds the Kolmogorov complexity of \emph{all} the real
roots of the polynomial.
\begin{lemma}\label{lem:rootenum}
  There is an algorithm \textsc{RootEnum} such that on input
  $(c_0,\dots,c_d)\in\Q^{d+1}$ and $r\in\N$, outputs a list of
  rationals $(q_1,\dots,q_\ell)$ of length at most $6d^2$ such that if
  $\hat{x}$ is a real root of $P(x)=\sum_{i=0}^d c_i x^i$, then there
  is an $i$, $1 \le i \le \ell$ such that $|q_i - \hat{x}|\le 2^{-r}$.
\end{lemma}
\begin{remark}
	\cref{alg:rootenum} (used in the proof of the above lemma) is a standalone result about root-finding when coefficients and the domain is only available as an approximation, and is possibly of independent interest. 
\end{remark}
Now, in order to bound the dimension of a point on the graph of a
polynomial $P(x)=\sum_{i=0}^d a_ix^i$, we try to bound the dimension
of the coefficients of a polynomial of equal degree, denoted as
$Q(x)=\sum_{i=0}^d b_ix^i$, \textit{almost coinciding} with $P$,
intersecting it at $x$. The coefficients of this polynomial will
provide sufficient information about the original polynomial, which in
turn will help estimating $x$.
\begin{theorem}\label{lem:errbound}
  Let $x\in\mathbb{R},\mathbf{a}\in\mathbb{R}^{d+1}$. For
  $r\geq t=-\log\|\mathbf{a}-\mathbf{b}\|$, and
  $\mathbf{b}\in B(\mathbf{a},1)$ where
  $\sum_{i=0}^da_ix^i=\sum_{i=0}^db_ix^i$, we have
  \begin{align*}
    K_r(\mathbf{b})\geq K_t(\mathbf{a})+K_{r-t,r}(x|\mathbf{a})+
    O(\log r).
  \end{align*}
\end{theorem}

\section{Spectra of Polynomials of degree $d\ge2$}
Recall the statement of the main theorem. 

\maintheorem*

Our proof follows the strategy of the work by Lutz and Stull \cite{Stull2020}. However, since we have to work with degree-$d$ polynomials, we deal with an increase in precision, as well as the presence of multiple roots. We indicate the strategy below, while simultaneously showing the similarity and emphasizing the differences from the work of Lutz and Stull \cite{Stull2020}. Broadly, the steps involved in the proof are as follows.

First, note that any polynomial $P(x)$ as defined above is completely described by the list of its real coefficients $(a_0,\dots,a_d)$. We show that any other polynomial as characterized by the list of coefficients $(b_0,\dots,b_d)$ which coincides with $P(x)$ must either be very close to $(a_0,\dots,a_d)$ in Euclidean distance, or must have very high Kolmogorov complexity. The technical steps in the proof, however, are radically different from the work of Lutz and Stull - since a degree-$d$ polynomial can have $d$ real roots, hence the intersection point of the polynomials is not uniquely specified by the lists of coefficients. Moreover, in order to compute the intersection points of the two polynomials, we employ a modification of the bisection method to find all real roots, where we have to manage the error introduced in this approximation, and possible multiplicities of real roots. 

Second, we show a lower bound for $K_r(b_0,\dots,b_d)$ in terms of $\|(a_0,\dots,a_d)-(b_0,\dots,b_d)\|$. 

Then, as in Lutz and Stull \cite{Stull2020}, we use an oracle $D$ such that the following sequences of inequalities hold.
\begin{align*}
K_r(x,P(x))
&\ge K^D_r(x,P(x))\\
&\ge K^D_r(x,a_0,\dots,a_d)-O(1)\\
&\ge K_r(x,a_0,\dots,a_d)-O(\log \log\|x\|)-O_x(\log r),
\end{align*}
establishing that $K_r(x,P(x))$ is not much less than
$K_r(x,a_0,\dots,a_d)$. 

We use these results to prove the required bound.

\subsection{Error estimates for polynomial approximation}

The following is a basic relation between $|a-b|$ and $|a^k-b^k|$,
$k\ge 1$.
\begin{lemma}\label{lem:degree_d_approx}
  If $a,b\in\R$ are such that $|a-b|<2^{-r}$, then for any integer
  $k \ge 1$, we have
\begin{align*}
\left|a^k - b^k\right| \ <\
2^{-r}k\times\max\{|a|,|a|^k,|b|,|b|^k\}. 
\end{align*}
\end{lemma}

\subsection{Lower bound for the complexity of $(x,P(x))$}

\begin{lemma}\label{lem:errorestimate}
  Let $\mathbf{a}=(a_0,\dots,a_d)\in\R^{d+1}$ with $a_d\ne 0$, $x\in\R$,
  precision parameter $r \in \N$, $\delta \in \R_+$, and parameters
  $\epsilon,\eta\in\Q_+$ be such that
  $r\ge 1+ \log(2\max\{|a_i|\mid 0\le i\le d\} +
  d(1+\max\{|x|,|x|^{2d}\}))$ and the following conditions hold.
\begin{enumerate}[label=(\roman*), itemsep=0pt, topsep=0pt]
\item $K_r(\mathbf{a})\le(\eta+\varepsilon)r$ and
\item For every $\mathbf{b}=(b_0,\dots,b_d)\in B(\mathbf{a},1)$ such
  that $x$ is a root of $\sum_{i=0}^d (b_i-a_i)x^i$, if
  $t=-\log||\mathbf{b}-\mathbf{a}||$ is at most $r$, then
\begin{align*}
K_r(\mathbf{b})\ge(\eta-\varepsilon)r + \delta(r-t).
\end{align*}
\end{enumerate}
Then, 
\begin{align}\label{ineq:lb_x_Px}
  K_r\left(x,P(x)\right)
  \ge
  K_r(\mathbf{a},x)
  - \frac{4\varepsilon}{\delta}(d+1)r
  -K(\varepsilon)-K(\eta)-O_{\mathbf{a}}(\log r). 
\end{align}
\end{lemma}
This is the analogue of Lemma 3.1 in Lutz and Stull \cite{Stull2020} generalized to degree $d$ polynomials. The outline of the proof is along the lines in their work, but with the error analysis modified for degree $d$ polynomials. 

The following lemma from \cite{Lutz2020e} is used to ensure one of the initial assumptions of \cref{lem:errorestimate}. 

\begin{lemma}[Lutz and Stull \cite{Lutz2020e}]\label{lem:oraclebound}
Let $r\in\mathbb{N}$, $z,\in\mathbb{R}^{d+1}$, and
$\eta\in\mathbb{Q}\cap[0,\dim(z)]$. Then, there is an oracle
$D=D(r,z,\eta)$ which satisfies  
\begin{itemize}
\item $K_t^D(z)=\min\left\{\eta r,K_t(z)\right\}+O(\log r)$, $\forall t\leq r$. 
\item $K_{t,r}^D(y|z)=K_{t,r}(y|z)+O(\log r)$, and
  $K_t^{z,D}(y)=K_t^D(y)+O(\log r)$, for all $m,t\in\mathbb{N}$, and
  $y\in\mathbb{R}^m$. 
\end{itemize}
\end{lemma}

\subsection{Proof of \cref{thm:d_degree_bound} }

Now, we proceed towards the proof of \cref{thm:d_degree_bound}.


\begin{proof}[Proof Sketch of \cref{thm:d_degree_bound}] 
The proof of this theorem follows the same steps as in the proof of the main theorem of Lutz and Stull \cite{Lutz2020e}, but with the bounds replaced appropriately. We summarize the argument below, highlighting the major steps. Let $\textbf{a}\subseteq\mathbb{R}^{d+1}$ be the coefficients of the polynomial. Denote $\dim(\mathbf{a})$ by $\rho$. Let $\epsilon\in\mathbb{Q}_+$, $\eta\in\left[0,\min\left\{\dim(\textbf{a}),\dim^\textbf{a}(x)\right\}\right]\cap\mathbb{Q}$, and $\delta=\dim^\textbf{a}(x)-\eta>0$. For each $r\in\mathbb{N}$, let $D_r$ be as defined in \cref{lem:oraclebound}. We show that the conditions of \cref{lem:errorestimate} hold for the choices of $\textbf{a},x,r,\delta,\epsilon,\eta$ made, which would eventually yield the desired result. 

By \cref{lem:oraclebound}, the oracle $D_r$ ensures that
$K_r^{D_r}(\textbf{a})\geq\eta r + O(\log r)$. This satisfies
condition (i) of \cref{lem:errorestimate}.

Now we show that condition (ii) of \cref{lem:errorestimate} also
holds, relative to $D_r$. Let $\textbf{b} \in B(\textbf{a},1)$ such
that $\sum_{i=0}^d b_i x^i = \sum_{i=0}^d a_i x^i$, $t = -\log
\|\textbf{b}-\textbf{a}\| \le r$. Therefore, 
\begin{align*}
K_r(\textbf{b})&\geq
K_t(\textbf{a})+K_{r-t,r}(x|\textbf{a})-O(\log r),&\cref{lem:errbound}\\
&\geq K_t^{D_r}(\textbf{a})+K_{r-t,r}(x|\textbf{a})-O(\log r),&\text{oracles never increase complexity}\\
&=\min\left\{\eta r,K_t(\textbf{a})\right\}+K_{r-t,r}(x|\textbf{a})-O(\log r),&\cref{lem:oraclebound}.1\\
&\geq\rho t-o(t)+(\delta+\eta)(r-t)-o(\log r),&\cref{lem:relative_conditional_complexity}\\
&\geq(\eta-\epsilon)r+\delta(r-t).
\end{align*}
Since both the conditions of \cref{lem:errorestimate} have been met, towards the final argument, we have, 
\begin{align*}
K_r\left(x,\sum_{i=0}^da_ix^i\right)&\geq K_r^{D_r}\left(x,\sum_{i=0}^da_ix^i\right),&\text{oracles never increase complexity}\\
&\geq K_r^{D_r}(\textbf{a},x)-\frac{4\epsilon}{\delta}(d+1)r-K(\epsilon)-K(\eta)-O_\textbf{a}(r),&\cref{lem:errorestimate}\\
&\geq K_r^{D_r}(x|\textbf{a})+K_r^{D_r}(\textbf{a})-\frac{4\epsilon}{\delta}(d+1)r-O(r)\\
&\geq K_r^{D_r}(x|\textbf{a})+\eta
  r-\frac{4\epsilon}{\delta}(d+1)r-O(\log r),&\cref{lem:oraclebound}.1\\
&\geq K_r(x|\textbf{a})+\eta r-\frac{4\epsilon}{\delta}(d+1)r-O(\log r).&\cref{lem:oraclebound}.2
\end{align*}
Taking $\liminf$ as $r\rightarrow\infty$ on both sides on the modified inequality, we get, 
\begin{align*}
\liminf_{r\rightarrow\infty}\frac{K_r\left(x,\sum_{i=0}^da_ix^i\right)}{r}&\geq\liminf_{r\rightarrow\infty}\frac{K_r(x|\mathbf{a})+\eta r-\frac{4\epsilon}{\delta}(d+1)r-O(\log r)}{r}\\
\implies\dim\left(x,\sum_{i=0}^da_ix^i\right)&\geq\dim(x|\mathbf{a})+\eta-\frac{4\epsilon}{\delta}(d+1). 
\end{align*}
Since $\eta,\gamma,\epsilon$ were chosen arbitrarily, we get, 
\begin{align*}
\dim\left(x,\sum_{i=0}^da_ix^i\right)&\geq\dim(x|\mathbf{a})+\min\{\dim(\mathbf{a}),\dim^\mathbf{a}(x)\}.\qedhere
\end{align*}
\end{proof}

\begin{corollary}
For almost every real $x \in \R$, we have 
\begin{align}
  \label{eqn:s_equals_one}
  \dim\left(x,\sum_{i=0}^da_ix^i\right)=1+\min\left\{\dim{(\mathbf{a}),1}\right\}.   
\end{align}
\end{corollary}
\begin{proof}
  We know that the set
  $\{x\in\mathbb{R}:x\text{ is \emph{not} Martin-L\"of random given
    $\mathbf{a}$}\}$ has measure 0. Hence, for almost every $x\in\R$,
  $\dim^\mathbf{a}(x)=1$. By
  \cref{lem:relative_conditional_complexity},
  $\dim^\mathbf{a}(x)\leq\dim(x|\mathbf{a})$. Thus, for almost every
  $x\in\R$, equation \ref{eqn:s_equals_one} holds.
\end{proof}

\section{Unit-length dimension spectrum for polynomials with
  $\dim(\mathbf{a})\leq1$}\label{sec:lo} 
This section deals with the polynomials, wherein
$\rho=\dim(\mathbf{a})\leq1$. Following is the main theorem of the section.

\maintheoremlow*

The basic idea is to construct a real $x\in\mathbb{R}$ by alternating
segments of a random point $y\in\mathbb{R}$ and approximations of the
coefficients $\mathbf{a}$. This interleaving is done in a
stage-by-stage manner, ensuring that $\dim(x)=s$. Simply adding random
bits to $x$ is not helpful, because the information content in the
point won't have any control on the information content of the
approximated function value. Instead, the constructed $x$ contains
information about the graph of the function $\sum_{i=0}^{d}a_ix^i$.

\textbf{Construction of $x$.} Let $y\in\mathbb{R}$ be Martin-L\"of random relative to $(\mathbf{a})$. For the stage $j=1$, let the \emph{stage length} be $h_j=2$. For any stage $j>1$, we define a sufficiently large \emph{stage length} $h_j$ by 
\begin{align}\label{eqn:stage_length}
	h_j&=\min\left\{h\ge2^{h_{j-1}}:K_h(\mathbf{a})\le\left(\rho+\frac{1}{j}\right)h\right\}.
\end{align}

Note that $h_j$ is finite, since $\dim(\mathbf{a})=\rho$ is finite. 

Denote the block length $\frac{h_j}{d}(1-s)$ by $i_j$. At stage $j$, we define the stretch $x[h_{j-1}+1\ldots h_j]$ of the binary expansion of $x$ by 
\begin{align}\label{def:x_case_1}
	x[r]&=\begin{cases}
		y[r]&h_{j-1}\le r<sh_j\\
		a_{(d+r\mod i_j)}[\floor{r/i_j}]&sh_j\le r<h_j,\ j\in\mathbb{N}
	\end{cases}
\end{align}

In other words, the first segment of $x$, i.e. $x[h_{j-1}\dots sh_j]$, is the same as $y[h_{j-1}\ldots sh_j]$, and the other segment $x[sh_j\ldots h_j]$ follows the interleaved pattern 
\begin{multline}a_1[0]a_2[0]\ldots a_d[0]\ \ 
	a_1[1]a_2[1] \ldots a_d[1]\ 
	\dots\\
	a_1[i_j-2]a_2[i_j-2]\ldots a_d[i_j-2]\ \ 
	a_1[i_j-1]a_2[i_j-1]\ldots a_d[i_j-1].
\end{multline}
It should be noted that the entire information of $\mathbf{a}$ has not been encoded into $x$. For instance, $a_0$ has not been used during the encoding. Only the required amount of information such that the complexity of the point on the polynomial can be reduced has been encoded. 

\begin{lemma}[Local Lipschitz condition]\label{lem:lipschitz}
	For a polynomial $P(x)\in\mathbb{R}[x]$, we have
	\begin{align*}
		\left(\exists c\in\mathbb{R}\right)\left(\forall x, y
          \in [0,1] \right)\left(|P(y)-P(x)|\leq c|y-x|\right),
	\end{align*} 
	where $c\in\mathbb{R}$ is a constant independent of $x$ or $y$.  
\end{lemma}
This lemma is especially useful to ensure minimal loss of precision, which will be clear in the subsequent lemmas. 

The following two lemmas show the effect of encoding the information of $\mathbf{a}$ into $x$ for every segment. These lemmas would be useful towards proving the main theorem. 

\begin{lemma}\label{lem:lemma10}
	For every $j\in\mathbb{N}$, and for $sh_j<r\leq h_j$, 
	\begin{align*}
		K_{\frac{r-sh_j}{d},r}\left(\mathbf{a}\middle|x,\sum_{i=0}^{d}a_ix^i\right)&\leq O(\log h_j).
	\end{align*}
\end{lemma}
Even though encoding $\mathbf{a}$ into $x$ helps reduce complexity, to control the information content, random bits (from $y$) were added. The following lemma shows that access to $\mathbf{a}$ doesn't provide any more information than what is already encoded in $x$, thereby showing the utility of adding segments from $y$. 
\begin{lemma}\label{lem:lemma11}
	For $j\in\mathbb{N}$, the following hold. 
	\begin{itemize}
		\item $K_t^\mathbf{a}(x)\geq t-O(\log h_j)$, $\forall t\leq sh_j$.
		\item $K_t^\mathbf{a}(x)\geq sh_j+t-O(\log h_j)$, $\forall h_j\leq t\leq sh_{j+1}$.
	\end{itemize}
\end{lemma}
We now try to obtain a lower bound on the complexity of the constructed point. There are two cases to be looked at; one over the segment of $x$ encoded by $\mathbf{a}$, and the other over the segment encoded by $y$. The following lemma talks about the segment of $x$ encoded with $\mathbf{a}$. 
\begin{lemma}\label{lem:lessleft}
	For every $\gamma>0$ and for large enough $j\in\mathbb{N}$,
	\begin{align*}
		K_r\left(x,\sum_{i=0}^{d}a_ix^i\right)&\geq\left(s+\rho-\gamma(d+1)^2\right)r, 
	\end{align*}
	for every $r\in(sh_j,h_j]$.
\end{lemma}
The previous lemma dealt with the segment encoded with $\mathbf{a}$. The following lemma talks about the other case, i.e. the segment encoded with $y$. 
\begin{lemma}\label{lem:lessright}
	For every $\gamma>0$ and for large enough $j\in\mathbb{N}$,
	\begin{align*} 	K_r\left(x,\sum_{i=0}^{d}a_ix^i\right)&\geq\left(s+\rho-\gamma(d+1)^2\right)r,
	\end{align*}
	for every $r\in(h_j,sh_{j+1}]$.
\end{lemma}
Now that we have a lower bound on the information of the point, the next lemma helps get an upper bound on the complexity of the point on the polynomial. 
\begin{lemma}\label{lem:lessup}
	For sufficiently large $j\in\mathbb{N}$,
	\begin{align*} K_{h_j}\left(x,\sum_{i=0}^{d}a_ix^i\right)&\leq\left(s+\rho\right)h_j.
	\end{align*}
\end{lemma}
Following is the main theorem of this section. 
\maintheoremlow*
\begin{proof}
	We have $\rho=\dim(\mathbf{a})\leq1$ and $s\in[0,1]$. \\
	\\
	For $s=0$, we have, 
	\begin{align*}
		K_{r/d}\left(\mathbf{a}\setminus
          a_0,\sum_{i=0}^{d}a_ix^i\right)&=K_{r/d}\left(\mathbf{a}\setminus
                                           a_0\right)+K_{r/d,r/d}\left(\sum_{i=0}^{d}a_ix^i\middle|\mathbf{a}\setminus
                                           a_0\right)+O(\log r).
        \end{align*}
        We have,
        \begin{align*}
          K_{\frac{r}{d},\frac{r}{d}}\left(\sum_{i=0}^d a_i x^i
          \middle| \mathbf{a} \setminus a_0\right)
          &=
          K_{\frac{r}{d},\frac{r}{d}}\left(a_0 \middle|
          \mathbf{a}\setminus a_0\right),&\text{by construction of $x$, since $s$=0}\\
          &\geq
            K_{\frac{r}{d}-\log c,r/d}\left(a_0\middle|\mathbf{a}\setminus a_0\right)+O(\log r)\\
          &\geq
            K_{\frac{r}{d},\frac{r}{d}}\left(a_0\middle|\mathbf{a}\setminus
            a_0\right)-K\left(\frac{r}{d}-\log c\right)\\
          &\quad\quad\quad\quad-\log c+O(\log
            r),&\cref{lem:precision_linear}\\ 
		&=K_{r/d,r/d}\left(a_0|\mathbf{a}\setminus a_0\right)+O(\log r)\\
		&=K_{r/d}(\mathbf{a})+O(\log r).
	\end{align*}
	which gives the desired result. 

For $s=1$, by \cref{thm:d_degree_bound}, for almost every point $x\in\mathbb{R}$ which is random relative to $(\mathbf{a})$, we have, $\dim\left(x,\sum_{i=0}^{d}a_ix^i\right)=1+\rho$. 

For $s\in[\rho,1)$, by \cref{thm:d_degree_bound}, by selecting $x$ to
be a point which satisfies $\dim(x)=\dim^{\mathbf{a}}(x)=s$, we get
that $\dim(x,\sum_{i=0}^d a_ix^i)=s+\rho$.

For $s\in(0,\rho)$, by \cref{lem:lessleft} and \cref{lem:lessright}, we
get,
	\begin{align*}
		\dim\left(x,\sum_{i=0}^{d}a_ix^i\right)&=\liminf_{r\rightarrow\infty}\frac{K_r\left(x,\sum_{i=0}^{d}a_ix^i\right)}{r}\\
		&\geq\liminf_{r\rightarrow\infty}\frac{\left(s+\rho-\gamma(d+1)^2\right)r}{r}\\
		&=s+\rho-\gamma(d+1)^2.
	\end{align*}
	Since $\gamma$ is arbitrarily chosen, we get
        $\dim\left(x,\sum_{i=0}^da_ix^i\right)=s+\rho$. Combining this
        with \cref{lem:lessup}, we get the desired result.
\end{proof}

\section{Width of the dimension spectra}

We provide some insights into the case where the dimension spectrum of
points on a polynomial could have diameter strictly greater than 1.
This leads to answer to another question posed by Stull
\cite{Stull2022}.

\begin{lemma}\label{lem:width}
  There is a class of polynomials of the form
  $P(x)=\sum_{i=0}^da_ix^i\in\mathbb{R}[x]$, with dimension spectrum
  having diameter strictly greater than 1.
\end{lemma}
\begin{proof}
  Denote the vector $(a_d,a_{d-1},\ldots,a_0)$ by $\mathbf{a}$.
  Consider the polynomial $P(x)=\sum_{i=0}^da_ix^i$. Let
  $S=\{(x,P(x)):x\in\mathbb{R}\}$ denote the graph of the polynomial.
  Let $\mathbf{a}\setminus \{a_0\}$ be Martin-L\"of random, and let
  $a_0\in\mathbb{R}$ be such that $\dim(a_0)=q$, where $q\in [0,1]$.
  Note that $\dim(\mathbf{a})>1$.
	
  Observe that $P(0)=a_0$, hence $\dim(0,P(0))=q$. Hence,
  $q\in sp(S)$.
	
  By \cite{Turetsky2011}, we know that $1\in sp(S)$.
	
  Next, let $x\in\mathbb{R}$ be such that it is Martin-L\"of random
  relative to $\mathbf{a}$. Then by \cref{thm:d_degree_bound}, we have
  $\dim(x,P(x)) \ge \dim(x|\mathbf{a}) +
  \min\{\dim(\mathbf{a}),\dim^{\mathbf{a}}(x)\} \ge 1+\min\{1,1\} =2$.
  Hence $\diam\ sp(S) \ge 2-q$. Thus, the diameters of the dimension
  spectrum of polynomials in this class lie in the range $[1,2]$.
\end{proof}

In the specific case of lines, we provide an answer to Stull's
question - there are lines with computable intercepts with dimension
spectrum greater than 1.

\begin{corollary}
  For the class of polynomials
  $\left\{P_i(x)\in\mathbb{R}[x]:\dim(P(0))=0\right\}_i$ with
  computable intercepts, $\diam\ sp(S)=2$.
\end{corollary}
\begin{proof}
  Consider $a_0$ to be a computable point (dimension 0). Then, in
  \cref{lem:width}, $q=0$, and hence $\diam\ sp=2$.
\end{proof}

\section{Open problems}

A natural question to consider is whether the methods in this work extend to spectra of arbitrary continuous curves. 

Building on our work, and Stull \cite{Stull2025}, we propose the dimension spectrum conjecture for high dimension polynomials. In other words, does the dimension spectrum of every high dimension polynomial also contain an interval of length 1?  

It is also interesting to determine whether there is a trigonometric polynomial whose dimension spectrum is a singleton. 

If every dimension level set in Euclidean space is path-connected, are such paths differentiable almost everywhere, or are there dimension level sets where every continuous path in them will be nowhere differentiable?


\bibliography{main_jabref_shared}

\section{Appendix}
This section contains proofs of all the intermediate lemmas and propositions, for the sake of completion whilst maintaining brevity. 
\begin{proof}[Proof of \cref{lem:strumseqeval}]
	The following algorithm (Algorithm \ref{alg:M1}) computes an ordered
	list of values obtained by evaluating the Sturm sequence of a
	polynomial $f(x)=\sum_{i=1}^dc_ix^i$ represented by the ordered list
	of coefficients $\mathbf{c}=(c_0, \dots, c_d)$ at the input value
	$\alpha$. Note that $f_1$, the derivative is equal to
	$\sum_{i=0}^{d-1} i c_{i} x^{i-1}$ which can be computed from the
	given list of coefficients $\textbf{c}$. Further, $f_i$, $i \ge 2$,
	can be obtained by the Euclidean algorithm.
	\begin{algorithm}[!h]
		\begin{algorithmic}[1]
			\caption{Algorithm to compute the Sturm Sequence}\label{alg:M1}
			\Procedure{\textsc{SturmSequence}}{$\mathbf{c},\alpha$} 
			\State $f(x)\gets\sum_{i=1}^dc_ix^i$  
			\State $\phi\gets\{\}$ \Comment{Sturm sequence of $f$} 
			\State $f_0\equiv f$ 
			\State $\phi\gets\phi\cup \{f_0\}$ 
			\State $f_1\equiv f'$  \Comment{Derivative of $f$}
			\State $\phi\gets\phi\cup \{f_1\}$ 
			\State $i\gets3$ 
			\While{$i\leq d$} 
			\State $f_{i-2}\equiv f_i\mod f_{i-1}$ \Comment{Euclidean
				division algorithm}  
			\State $\phi\gets\phi\cup \{f_i\}$ 
			\State $i\gets i+1$ 
			\EndWhile 
			\Statex
			\State $\theta\gets()$ \Comment{Evaluation list} 
			\State $i\gets0$
			\While{$i\leq d$} 
			\State $\theta\gets\theta \catenate (f_i(\alpha))$
			\Comment{append to the list $\theta$}  
			\State $i\gets i+1$
			\EndWhile 
			\State \textbf{return} $\theta$ 
			\EndProcedure
		\end{algorithmic}
	\end{algorithm}
\end{proof}
\begin{proof}[Proof of \cref{lem:signchangecalc}]
	We observe that computing the exact sign of a polynomial at a point
	is not possible. This is because the evaluation point can be close
	enough to one of the roots given the precision limitation, in which
	case the approximation may have opposite sign of the actual value.
	The following algorithm takes care of this issue and ensures that no
	root is missed while not too many additional points are added.
	
	Let $\textsc{MaxSignChange}$ be a Turing machine, described by
	\cref{alg:MaxSignChange}, that upper bounds the possible number of
	sign changes in an ordered list of input values. In the first part
	of \cref{alg:MaxSignChange}, we compute the number of sign changes
	in the input sequence without considering the magnitude of the
	corresponding values. This simple counting might result in some
	roots getting missed because of the precision limitations.
	
	If the absolute value of the input is less than the precision range,
	i.e. $|\theta_i|<\gamma$ for some $i$, the computed sign may not be
	the actual sign. To take care of such cases, in the second part of
	\cref{alg:MaxSignChange}, for every input value lying within the
	precision range, instead of considering the computed sign, we
	consider both possible signs, and compute the maximum number of
	possible sign changes that could have had taken place. The number of
	additional sign changes that could take place depends on the signs
	of the preceding and following values in the input sequence. If the
	preceding and following terms have opposite signs, then for any
	possible sign of the current term, the number of sign changes will
	remain the same. For the case where the preceding and following
	terms have the same signs, a different sign of the current term
	could result in the 2 additional sign changes which would previously
	not have been counted (which happens when all the three consecutive
	terms have same sign). Hence, the value returned by
	$\textsc{MaxSignChange}$ is an overestimate of the actual number of
	sign changes occurring in the input sequence.
	
	Note that to get an upper bound on the number of roots in an
	interval $(a,b]$, by \cref{thm:sturm}, we need to upper bound the
	number of sign changes at $a$ and simultaneously lower bound the
	number of sign changes at $b$. By replacing $s\gets s+1$ by
	$s\gets s-1$ in lines \ref{algline:case0} and \ref{algline:casen},
	and replacing $s\gets s+2$ by $s\gets s-2$ in line
	\ref{algline:case1} of \cref{alg:MaxSignChange}, we get the dual
	algorithm $\textsc{MinSignChange}$, which lower bounds the number of
	sign changes in the input sequence.
	\begin{algorithm}[!h]
		\caption{\textsc{MaxSignChange}}\label{alg:MaxSignChange}
		\begin{algorithmic}[1]
			\Procedure{\textsc{MaxSignChange}}
			{$\theta=(\theta_0,\theta_1,\ldots,\theta_n)$,~$\gamma$}
			\State $s\gets0$ \Comment{Number of sign flips} 
			\State $i\gets0$
			\While{$i\leq n-1$}
			\If{$sgn(\theta_i\theta_{i+1})<0$} 
			\State $s\gets s+1$ 
			\EndIf
			\State $i\gets i+1$
			\EndWhile
			\Statex
			\State $i\gets0$
			\If{$|\theta_i|<2^{-r}$}
			\State $s\gets s+1$\label{algline:case0}
			\EndIf
			\State $i\gets i+1$
			\Statex
			\While{$i\leq n-1$}
			\If{$|\theta_i|<2^{-r}$}
			\If{$sgn(\theta_{i-1}\theta_{i+1})>0$}
			\State $s\gets s+2$\label{algline:case1}
			\EndIf
			\EndIf
			\State $i\gets i+1$
			\EndWhile
			\Statex
			\If{$|\theta_i|<\gamma$}
			\State $s\gets s+1$\label{algline:casen}
			\EndIf
			\State \textbf{return} $s$
			\EndProcedure
		\end{algorithmic}
	\end{algorithm}
\end{proof}
\begin{proof}[Proof of \cref{lem:rootenum}]
	\cref{alg:rootenum} helps enumerate the roots of the polynomial
	under consideration in a sorted manner. Since the number of roots is
	bounded, the description of the required root can be encoded in a
	corresponding program.
	\begin{algorithm}[H]
		\caption{Root enumeration}\label{alg:rootenum}
		\begin{algorithmic}[1]
			\Procedure{RootEnum}{$\mathbf{c},r$}
			\State Compute $\beta$ for $f(x)=\sum_{i=1}^dc_ix^i$
			\Comment{Cauchy bound}
			\State $\Delta\gets\{-\beta,\beta\}$ \Comment{Grid elements}
			\State $\theta\gets\{\}$ \Comment{Sorted list of roots}
			\State $r\gets r+1+\log\beta$\Comment{To ensure $r$ bits of
				precision after decimal point} 
			\While{$r\geq0$}
			\While{$i\not=j\in\Delta$}
			\State $\Delta\gets\Delta\cup\left\{\frac{i+j}{2}\right\}$
			\EndWhile
			\State $r\gets r-1$
			\EndWhile
			\Statex
			\State $\textit{sort }(\Delta)$
			\State $z\gets-\beta$
			\While{$z\in\Delta$}
			\State $L \gets
			\textsc{MaxSignChange}\left(\textsc{SturmSequence}\left(\mathbf{c},z\right),c_p
			1      2^{-dr} \right)$
			\State $R \gets
			\textsc{MinSignChange}\left(\textsc{SturmSequence}\left(\mathbf{c},z+\frac{\beta}{2^{r-1}}\right),
			c_p 2^{-dr}\right)$
			\If{$L-R\geq1$}
			\State $\theta\gets\theta\cup \{z+\frac{\beta}{2^r}\}$
			\EndIf
			\State $z\gets z+\frac{\beta}{2^{r-1}}$
			\EndWhile
			\State \textbf{return} $\theta$
			\EndProcedure
		\end{algorithmic}
	\end{algorithm}
	In the second part of \cref{alg:rootenum}, after $r$ iterations, the
	following cases may arise. By Sturm's theorem, if the difference in
	the number of sign changes at the ends of an interval is $0$, there
	is no root in that interval, and hence the value is not included in
	the output list. For the other case, since we have a precision of
	$r$ bits, the roots are close enough already, and hence we can pick
	any value in the interval as a representative of the roots
	regardless of the number of roots present. This might result in a
	slightly smaller output list, but for every root, we have an
	appropriate approximation in that interval.
	
	If we know the exact coefficients, Sturm's theorem gives at most
	$d$ roots. To upper bound the length of the output list $\theta$,
	note that if in lines 13-15, $L-R$ differs from the actual number of
	sign changes for the interval, then at least one of the polynomials
	in the Sturm sequence for at least one of the endpoints is less than
	or equal to $c_p 2^{-dr}$. Here, $c_p$ is a constant depending only
	on $P$. By \cref{lem:poly_bound}, this happens only when one of the
	endpoints is within $2^{-r}$ of some root, for all sufficiently
	large $r$. There are at most $d$ polynomials in the Sturm sequence
	at each endpoint, hence $2d$ polynomials, each with at most $d$
	roots. Each root $\alpha$ can have at most 2 distinct neighbors $x$
	other than itself, for which $|p(x)| < c_p 2^{-dr}$. Thus the length
	of the output list $\theta$ in the above algorithm is at most $6d^2$.
\end{proof}
\begin{proof}[Proof of \cref{lem:errbound}]
	Let $P(x)=\sum_{i=0}^d a_i x^i$, and $Q(x)=\sum_{i=0}^d b_i x^i$ be
	the given $d$-degree polynomials.
	
	By the symmetry of information (\cref{lem:soi}), we get,
	\begin{align*}
		K_r(\mathbf{b})\geq
		K_{r,r}(\mathbf{b}|\mathbf{a})~+~
		K_r(\mathbf{a})-K_{r,r}(\mathbf{a}|\mathbf{b})~-~
		O_{\mathbf{a}}(\log r).
	\end{align*}
	We bound the terms on the right to estimate the information in the
	polynomial $C(x)=P(x)-Q(x)$.
	
	First, we bound the term
	$K_r(\mathbf{a})-K_{r,r}(\mathbf{a}|\mathbf{b})$.
	
	Since $\|\mathbf{b}-\mathbf{a}\| < 2^{-r}$, we have
	$\left(\forall r\geq t\right)\left(B(\mathbf{b},2^{-r})\subseteq
	B(\mathbf{a},2^{-(t-1)})\right)$. Thus,
	$K_{r,r}(\mathbf{a}|\mathbf{b})\leq K_{r,t-1}(\mathbf{a}|\mathbf{a})$.
	Thus,
	\begin{align*}
		K_r(\mathbf{a})-K_{r,r}
		(\mathbf{a}|\mathbf{b})\geq K_r(\mathbf{a})-K_{r,t-1}(\mathbf{a}|\mathbf{a}).
	\end{align*}
	By \cref{lem:precision_var}, we get, 
	\begin{align*}
		K_r(\mathbf{a})-K_{r,r}(\mathbf{a}|\mathbf{b})\geq K_{t-1}(\mathbf{a})-O(\log r).
	\end{align*}
	Now, by \cref{lem:precision_linear}, we get, 
	\begin{align}
		\label{ineq:errorbound_one}
		K_r(\mathbf{a})-K_{r,r}(\mathbf{a}|\mathbf{b})\geq K_t(\mathbf{a})-O(\log r).
	\end{align}

	To lower bound $K_{r,r}(\mathbf{b}|\mathbf{a})$ such that at $x\in\R$,
	we have $\sum_{i=0}^d a_i x^i = \sum_{i=0}^d b_i x^i$, denote
	$\mathbf{b}=(b_0,\dots,b_{d})$ and consider a program $\pi$ such
	that $U(\pi,\tilde{\mathbf{a}}) = \mathbf{b}$, where
	$\tilde{\mathbf{a}} \in B(\mathbf{a},2^{-r})\cap\Q^{d+1}$. We then
	run \cref{alg:rootenum} on inputs $\mathbf{b}$ and $r \in \N$. We
	know that the output list has some $\tilde{x}\in\Q$ which
	approximates $x$ to precision $r$.
	
	Thus, by the bound in \cref{lem:rootenum}, we know that
	$|x - \tilde{x}| \le 2^{-r}$. Since the list output by the algorithm
	has at most $6d^2$ entries, in order to specify a particular root in
	the output list of \cref{alg:rootenum}, we need at most $O(\log d)$
	bits. Hence, we have
	\begin{align*}
		K_{r,r}(x|\mathbf{a}) \le 
		K_{r,r}(\mathbf{b} | \mathbf{a}) + 
		O_{x,\mathbf{a}}(\log r) + 
		O_{x,\mathbf{a}}(\log d).
	\end{align*}
	Since $d$ is a constant independent of $r$, we have
	\begin{align}
		\label{ineq:errorbound_two}
		K_{r,r}(\mathbf{b}|\mathbf{a}) \ge K_{r,r}(x|\mathbf{a}) 
		- O_{x,\mathbf{a}}(\log r)
	\end{align}
	whence it follows that
	\begin{align*}
		K_r(\mathbf{b}) &\geq K_t(\mathbf{a}) + K_{r,r}(x|\mathbf{a}) -
		O_{x,\mathbf{a}}(\log r)\\
		&\geq K_t(\mathbf{a}) + K_{r-t,r}(x|\mathbf{a}) -
		O_{x,\mathbf{a}}(\log r).
		\quad\qedhere
	\end{align*}
\end{proof}
\begin{proof}[Proof of \cref{lem:errorestimate}]
	First, we define an oracle Turing machine $M$ with oracle $A$ such
	that, given input $\sigma=\xi\alpha\rho\epsilon\tau$ where
	$U(\xi)=(\snt x ,\snt y
	)\in\Q^2$,$U(\alpha)=(\hat{a_0},\dots,\hat{a_d})\in\Q^d$,
	$U(\rho)=r$,$U(\epsilon)=\varepsilon$ and $U(\tau)=\eta$, does the
	following. It executes all programs having length at most
	$(\varepsilon+\eta)r$ in parallel, and if one of the programs
	outputs $\snt {\mathbf{a}} \in B\left(\mathbf{a},2^{-1}\right)$ and such that
	\begin{align*}
		|\snt y-\left(\snt P\left(\snt x\right)\right)|<
		\left(d+1\right)2^{-r} \left(\tilde{t}_d +
		d\tilde{t}_d\max\left\{|\tilde{a_i}|\mid 0\le i \le d\right\} \right).
	\end{align*}
	outputs $(\snt{\mathbf{a}},\snt x)$, where 
	$\tilde{t}_d =
	\max\left\{|\tilde{x}|+2^{-r},\left(|\tilde{x}|+2^{-r}\right)^d\right\}$.
	
	We first argue that $M$ halts. Note that $M$ has, as input, a rough
	approximation $(\hat{a_0},\dots,\hat{a_d})$ within $B\left((a_0,\ldots,a_d),2^{-1}\right)$.
	By assumption (i) and since $r$ is sufficiently high, there is some
	rational vector $(\tilde{a_0},\dots,\tilde{a_d})$ inside
	$B\left((a_0,\ldots,a_d),2^{-r}\right)$ with complexity at most $(\eta+\varepsilon)r$. Now,
	$(\snt x, \snt y) \in B((x,y),2^{-r})$. Hence,
	\begin{align*}
		\left|\tilde{y}-\left(\sum_{i=0}^d \tilde{a_i}x^i\right)\right|
		&\le
		\left|\sum_{i=0}^d a_ix^i- \sum_{i=0}^d\tilde{a_i}\tilde{x}^i\right|+ 
		\left|\tilde{y}-\left(\sum_{i=0}^d a_ix^i\right)\right|\\
		&\le
		\left|\sum_{i=0}^d a_ix^i- \sum_{i=0}^d\tilde{a_i}\tilde{x}^i\right|+ 
		2^{-r}. 
	\end{align*}
	We have, for $0\le i\le d$,
	\begin{align*}
		|a_ix^i-\tilde{a_i}\tilde{x}^i|
		&\le|a_ix^i-\tilde{a_i}x^i|+|\tilde{a_i}x^i-\tilde{a_i}\tilde{x}^i|\\
		&\le|(a_i-\tilde{a_i}) x^i|+|\tilde{a_i}(x^i-\tilde{x}^i)|\\
		&\le2^{-r}|x^i|+|\tilde{a_i}(x^i-\tilde{x}^i)|\\
		&\le2^{-r}|x^i|+|\tilde{a_i}|2^{-r}\times i\times 
		\max\left\{|x|,|x|^i,|\tilde{x}|,|\tilde{x}^i|\right\}.
	\end{align*}
	where the last inequality follows by \cref{lem:degree_d_approx}. Since
	$|\tilde{x}|<|x|+2^{-r}$, we have
	$|\tilde{x}^i|<|x|^i\times i\times2^{-r}$.
	
	Since the machine has access only to $|\tilde{a}|$ and $|\tilde{x}|$,
	we now form an upper bound in terms of these. If $|x|\le 1$, then
	$|x|^i \le |x| \le |\tilde{x}|+2^{-r}$. Otherwise, when $|x|>1$, we
	have $|x|\le |\tilde{x}|+2^{-r}$, so $|x|^i \le (|x|+2^{-r})^i$.
	
	Hence, letting $\tilde{t}_i =
	\max\{|\tilde{x}|+2^{-r},(|\tilde{x}|+2^{-r})^i\}$, we have
	\begin{align*}
		|a_ix^i-\tilde{a_i}\tilde{x}^i|
		&\le
		2^{-r}\tilde{t}_i +
		i|\tilde{a_i}|2^{-r}\tilde{t}_i
	\end{align*}
	implying that
	\begin{align*}
		\sum_{i=0}^d |a_ix^i-\tilde{a_i}\tilde{x}^i|
		\le
		(d+1)2^{-r} \left(\tilde{t}_d +
		d\tilde{t}_d\max\left\{|\tilde{a_i}|\mid 0\le i \le d\right\} \right).
	\end{align*}
	This shows that $M$ halts.

	Second, we show that we can apply condition (ii) in the statement of
	the lemma. To show this, we construct a tuple $\left(b_0,\dots,b_d\right)\in B\left(\left(a_0,\dots,a_d\right),1\right)$ such that
	$\sum_{i=0}^d\left(a_i-b_i\right)x^i=0$. Consider the polynomial defined by setting
	the coefficients $b_i=\tilde{a_i}$ for $1 \le i \le d$ and
	$$b_0 =
	\left(\sum_{i=0}^da_ix^i\right)-\left(\sum_{i=1}^d
	\tilde{a_i}x^i\right).$$ We can easily verify that
	$\sum_{i=0}^d b_i x^i = \sum_{i=0}^d a_ix^i$. In order to apply
	condition (ii), it remains to show that
	$(b_0,\dots,b_d)\in B\left((a_0,\dots,a_d),1\right)$. We have
	\begin{align}
		\notag
		\|(a_0,\dots,a_d)-(b_0,\dots,b_d)\|^2
		&=\left(\sum_{i=0}^d \left(a_i-b_i\right)^2\right)\\
		\notag
		&=\left(\sum_{i=1}^d \left(a_i-b_i\right)^2\right)+
		(b_0-a_0)^2\\
		\notag
		&=\left(\sum_{i=1}^d \left(a_i-\tilde{a_i}\right)^2\right)+
		(b_0-a_0)^2\\
		\label{ineq:L2_difference}
		&\le2^{-2r}d+(b_0-a_0)^2.
	\end{align}
	We now have
	\begin{align*}
		b_0-a_0 
		&=\left(\sum_{i=0}^d a_i x^i\right)-
		\left(\sum_{i=0}^d \tilde{a_i} x^i\right)-a_0\\
		&=\left(\sum_{i=1}^d(a_i-\tilde{a_i} x^i\right)\\
		&\le2^{-r}d \max\{|x|,|x|^d\}.
	\end{align*}
	Thus,
	\begin{align*}
		(b_0-a_0)^2 \le 2^{-2r}d^2\max\{|x|,|x|^{2d}\}.
	\end{align*}
	Using this bound in \Cref{ineq:L2_difference}, we obtain
	\begin{align}
		\label{ineq:L_final}
		\|(b_0,\dots,b_d)-(a_0,\dots,a_d)\|^2
		&<d^2 2^{-2r} \left(1+\max \{|x|,|x|^{2d}\}\right).
	\end{align}
	Since $r>\log ( d(1+\max\{|x|,|x|^{2d}\}) )$, we have 
	\begin{align*}
		\frac{d^2(1+\max\{|x|,|x|^{2d}\})}{2^{2r}} < 1,
	\end{align*}
	as required. This shows that we can apply condition (ii). The
	remainder of the estimation is similar to that of Lemma 3.1 in N. Lutz
	and Stull \cite{Lutz2020e}, which we now summarize.
	
	Now we establish the lower bound for $K_r(x,P(x))$ using conditions
	(i) and (ii). Set $\gamma=\log(2\max\{a_i\mid 0\le i\le d\} +
	d(1+\max\{|x|,|x|^{2d}\})$. By inequality \eqref{ineq:L_final}, we know that
	\begin{align}
		\notag
		K_{r-\gamma-1}(b_0,\dots,b_d,x) 
		&\le \ell(\pi)+O(1)\\
		\notag
		&\le K_r(x,P(x))+K_2(a_0,\dots,a_d)+K(r)+K(\varepsilon)+K(\eta)+O(1)\\
		\label{ineq:lb_x_Px_final}
		&=
		K_r(x,P(x))+K(\varepsilon)+K(\eta)+O_{a_0,\dots,a_d}(\log r).
	\end{align}
	where the last inequality follows by condition (i). 
	
	If $r<t$, then by \cref{lem:precision_linear} we get
	\begin{align*}
		K_r(b_0,\dots,b_d,x)\ge K_r(a_0,\dots,a_d,x)-O_{a,b,x}(\log r),
	\end{align*}
	which, by \cref{ineq:lb_x_Px_final} establishes
	\begin{align*}
		K_r(x,P(x))\ge K_r(a_0,\dots,a_d,x)-K(\varepsilon)-K(\eta)-O_{a_0,\dots,a_d}(\log
		r).
	\end{align*}
	
	Otherwise, if $t<r$, by \cref{lem:precision_linear}, we have 
	\begin{align*}
		K_r(\mathbf{b},x)\ge K_r(\mathbf{a},x) - (d+1)(r-t)-O_{a_0,\dots,a_d,x}(\log r).
	\end{align*}
	By an argument similar as in Lemma 3.1 of (i), it is possible to show that  
	\begin{align*}
		r-t\le\frac{2\varepsilon r}{\delta}+O_{a_0,\dots,a_d,x}(\log r),
	\end{align*}
	establishing \cref{ineq:lb_x_Px}. 
\end{proof}


\begin{proof}[Proof of \cref{lem:lemma10}]
	By \cref{lem:soi}, we have, 
	\begin{align*}
		K_{\frac{r-sh_j}{d},r}\left(\mathbf{a}\middle|x,\sum_{i=0}^{d}a_ix^i\right)=&K_{\frac{r-sh_j}{d},r}\left(\mathbf{a}\setminus a_0\middle|x,\sum_{i=0}^{d}a_ix^i\right)+\\
		&K_{\frac{r-sh_j}{d},r}\left(a_0\middle|x,\mathbf{a}\setminus a_0,\sum_{i=0}^{d}a_ix^i\right)+O(\log r).
	\end{align*}
	By construction of $x$, the last $r-sh_j$ bits of $x$ contain the first $\frac{h_j}{d}(1-s)$ bits of each $a\in\mathbf{a}\setminus a_0$ in an interleaved fashion. Hence, 
	\begin{align*}
		K_{\frac{r-h_j}{d},r}\left(\mathbf{a}\setminus a_0\middle|x,\sum_{i=0}^{d}a_ix^i\right)&\leq K_{\frac{r-sh_j}{d},r}(\mathbf{a}|x)\\
		&\leq O(\log h_j).
	\end{align*}
	Now, given sufficient number of approximations of $x$, $\mathbf{a}\setminus a_0$, and $\sum_{i=0}^{d}a_ix^i$, we can approximate $a_0$. 
	
	Hence, by \cref{lem:lipschitz}, 
	\begin{align*}
		K_{\frac{r-sh_j}{d}-\log c,r}\left(a_0\middle|x,\mathbf{a}\setminus a_0,\sum_{i=0}^{d}a_ix^i\right)&\leq O(\log h_j).
	\end{align*}
	By \cref{lem:precision_linear}, we get, 
	\begin{align*}
		K_{\frac{r-sh_j}{d},r}\left(a_0\middle|x,\mathbf{a}\setminus a_0,\sum_{i=0}^{d}a_ix^i\right)&\leq K_{\frac{r-sh_j}{d}-\log c,r}\left(a_0\middle|x,\mathbf{a}\setminus a_0,\sum_{i=0}^{d}a_ix^i\right)+K\left(\frac{r-sh_j}{d}-\log c\right)+O(1)\\
		&\leq K_{\frac{r-sh_j}{d}-\log c,r}\left(a_0\middle|x,\mathbf{a}\setminus a_0,\sum_{i=0}^{d}a_ix^i\right)+O(\log h_j)\\
		&\leq O(\log h_j). 
	\end{align*}
	Therefore, by the initial result, we get 
	\begin{align*}
		K_{\frac{r-sh_j}{d},r}\left(\mathbf{a}\middle|x,\sum_{i=0}^{d}a_ix^i\right)&\leq O(\log h_j). 
	\end{align*} 
\end{proof}
\begin{proof}[Proof of \cref{lem:lemma11}]
	For the first statement, let $t<sh_j$. By construction of $x$, we have, 
	\begin{align*}
		K_t^\mathbf{a}(x)&\geq K_t^\mathbf{a}(y)-h_{j-1}-O(\log t)\\
		&\geq t-O(\log t)-\log h_j-O(\log t)\\
		&\geq t-O(\log h_j).
	\end{align*}
	For the second statement, let $h_j\leq t\leq sh_{j+1}$. We have, 
	\begin{align*}
		K_t^\mathbf{a}(x)&=K_{h_j}^\mathbf{a}(x)+K_{t,h_j}^\mathbf{a}(x)-O(\log t),&\textit{ by \cref{lem:soi}}\\
		&\geq sh_j+K_{t,h_j}^\mathbf{a}(x)-O(\log t),&\textit{by \cref{lem:lemma11}.1}\\
		&\geq sh_j+K_{t,h_j}^\mathbf{a}(y)-O(\log t)\\
		&\geq t-h_j+sh_j-O(\log t)\\
		&\geq sh_j+t-O(\log t). 
	\end{align*}
	which completes the proof. 
\end{proof}
\begin{proof}[Proof of \cref{lem:lessleft}]
  Let $\eta\in\mathbb{Q}$ such that $\rho-\frac{\gamma}{4}(d+1)^2<\eta<\rho-\gamma^2$, and $\epsilon\in\mathbb{Q}$ such that $\epsilon<\frac{\gamma}{16}(\rho-\eta)$. Let $D=D(r,\mathbf{a},\eta)$ be the oracle for \cref{lem:oraclebound}. \\
  \\
  For a close by polynomial $\sum_{i=0}^db_ix^i\in\mathbb{R}[x]$ such
  that $t:=\|\mathbf{a}-\mathbf{b}\|\geq r-sh_j$, and
  $\sum_{i=0}^{d}a_ix^i=\sum_{i=0}^{d}b_ix^i$, by \cref{lem:errbound},
  \cref{lem:oraclebound}, and \cref{lem:lemma11}, we have
	\begin{align*}
          K_r(\mathbf{b})
          &\geq K_t(\mathbf{a}) + K_{r-t,t}(x|\mathbf{a})-O(\log r)\\
          &\geq K_t^D(\mathbf{a})+K_{r-t,t}^D(x|\mathbf{a})-O(\log r)\\
          &\geq K_t^D(\mathbf{a})+K_{r-t,t}(x|\mathbf{a})-O(\log r)\\
          &\geq K_t^D(\mathbf{a})+r-t-O(\log r).
	\end{align*}
	To ensure the conditions for \cref{lem:oraclebound}, we have 2 cases, $K_t^D(\mathbf{a})=\eta r$, and $K_t^D(\mathbf{a})=K_t(\mathbf{a})$. 
	
	For the case $K_t^D(\mathbf{a})=\eta r$, we have, 
	\begin{align*}
		K_r(\mathbf{b})&\geq\eta r+r-t-O(\log r)\\
		&\geq\left(\eta-\epsilon\right)r+r-t\\
		&\geq(\eta-\epsilon)r+(1-\eta)(r-t).&\textit{since }\textit{ $\eta<1$}
	\end{align*}
	
	For the case $K_t^D(\mathbf{a})=K_t(\mathbf{a})$, we have, 
	\begin{align*}
		K_r(\mathbf{b})&\geq K_t(\mathbf{a})+r-t-O(\log r)\\
		&\geq\rho t-o(t)+r-t-O(\log r),&\textit{ definition of $\dim$}\\
		&\geq\eta r+(1-\eta)r-t(1-\rho)-\epsilon r\\
		&\geq\eta r-\epsilon r+(1-\eta)(r-t),&\textit{since }\textit{ $\eta<\rho$}\\
		&=(\eta-\epsilon)r+(1-\eta)(r-t).
	\end{align*}
	Hence, for oracle $D$, with $\delta=1-\eta$, by \cref{lem:errorestimate}, we have, 
	\begin{align*}
		K_r\left(x,\sum_{i=0}^{d}a_ix^i\right)&\geq K_r(\mathbf{a},x)-\frac{4\epsilon}{\delta}(d+1)r-K(\epsilon)-K(\eta)-O_\mathbf{a}(\log r)\\
		&\geq K_r^D(\mathbf{a},x)-\frac{4\epsilon}{\delta}(d+1)r-K(\epsilon)-K(\eta)-O_\mathbf{a}(\log r)\\
		&\geq K_r^D(\mathbf{a},x)-K_{\frac{r-sh_j}{d},r}\left(\mathbf{a}\middle|x,\sum_{i=0}^{d}a_ix^i\right)-\frac{4\epsilon}{\delta}(d+1)r-K(\epsilon,\eta)-O(\log r)\\
		&\geq K_r^D(\mathbf{a},x)-K_{\frac{r-sh_j}{d},r}\left(\mathbf{a}\middle|x,\sum_{i=0}^{d}a_ix^i\right)-\frac{4\epsilon}{1-\eta}(d+1)r-K(\epsilon,\eta)-O(\log r)\\
		&\geq K_r^D(\mathbf{a},x)-K_{\frac{r-sh_j}{d},r}\left(\mathbf{a}\middle|x,\sum_{i=0}^{d}a_ix^i\right)-\frac{\gamma}{4}(d+1)r-O(1)-O(\log r)\\
		&\geq K_r^D(\mathbf{a},x)-K_{\frac{r-sh_j}{d},r}\left(\mathbf{a}\middle|x,\sum_{i=0}^{d}a_ix^i\right)-\frac{\gamma}{4}(d+1)^2r-\frac{\gamma}{8}(d+1)^2r\\
		&\geq K_r^D(\mathbf{a},x)-K_{\frac{r-sh_j}{d},r}\left(\mathbf{a}\middle|x,\sum_{i=0}^{d}a_ix^i\right)-\frac{3\gamma}{8}(d+1)^2r.
	\end{align*}
	Now, by symmetry of information, we have, 
	\begin{align*}
		K_r^D(\mathbf{a},x)&=K_r^D(\mathbf{a})+K_{r,r}^D(x|\mathbf{a})-O(\log r)\\
		&=K_r^D(\mathbf{a})+K_{r,r}(x|\mathbf{a})-O(\log r),&\textit{ by \cref{lem:oraclebound}.1}\\
		&\geq\eta r+K_{r,r}(x|\mathbf{a})-O(\log r),&\textit{ by \cref{lem:oraclebound}.2}\\
		&\geq\eta r+sh_j-O(\log r),&\textit{ by \cref{lem:lemma11}}\\
		&\geq\eta r+sh_j-\frac{\gamma}{4}(d+1)^2r.
	\end{align*}
	By \cref{lem:lemma10}, we have,
	\begin{align*}
		K_{\frac{r-sh_j}{d},r}\left(\mathbf{a}\middle|x,\sum_{i=0}^{d}a_ix^i\right)&\leq\ O(\log r)\\
		&\leq\frac{\gamma}{8}(d+1)^2r.
	\end{align*}
	Thus, by combining the above results, we get,
	\begin{align*} 
		K_r\left(x,\sum_{i=0}^{d}a_ix^i\right)&\geq K_r^D(\mathbf{a},x)-K_{\frac{r-sh_j}{d},r}\left(\mathbf{a}\middle|x,\sum_{i=0}^{d}a_ix^i\right)-\frac{3\gamma}{8}(d+1)^2r\\
		&\geq\eta r+sh_j-\frac{\gamma}{4}(d+1)^2r-\frac{\gamma}{8}(d+1)^2r-\frac{3\gamma}{8}(d+1)^2r\\
		&\geq\rho r-\frac{\gamma}{4}(d+1)^2r+sh_j-\frac{3\gamma}{4}(d+1)^2r\\
		&\geq\left(s+\rho-\gamma(d+1)^2\right)r.
	\end{align*}
\end{proof}
\begin{proof}[Proof of \cref{lem:lessright}]
	Let $\hat{s}\in\mathbb{Q}\cap(0,s)$ such that $\frac{\gamma}{8}\leq s-\hat{s}\leq\frac{\gamma}{4}(d+1)^2$, and $\hat{\rho}\in\mathbb{Q}\cap(0,\rho)$ such that $\frac{\gamma}{8}\leq \rho-\hat{\rho}\leq\frac{\gamma}{4}(d+1)^2$.\\
	\\
	Define $\alpha:=\frac{s(r-h_j)+\rho h_j}{r}$, and $\eta\in\mathbb{Q}\cap(0,\alpha)$ as $\eta:=\frac{\hat{s}(r-h_j)+\hat{\rho}h_j}{r}$.\\
	\\
	Observe that 
	\begin{align*}
		\alpha-\eta&=\frac{s(r-h_j)+\rho h_j-\hat{s}(r-h_j)-\hat{\rho}h_j}{r}\\
		&=\frac{(s-\hat{s})(r-h_j)+(\rho-\hat{\rho})h_j}{r}\\
		&\geq\frac{1}{r}\left(\frac{\gamma}{8}(r-h_j)+\frac{\gamma}{8}h_j\right)\quad=\frac{\gamma}{8}.
	\end{align*}
	Similarly,
	\begin{align*}
		\alpha-\eta
		&=\frac{(s-\hat{s})(r-h_j)+(\rho-\hat{\rho})h_j}{r}\\
		&\leq\frac{1}{r}\left(\frac{\gamma(d+1)^2}{4}\left(r-h_j\right)+\frac{\gamma(d+1)^2}{4}h_j\right)
		\quad=\frac{\gamma}{4}(d+1)^2.
	\end{align*}
	Lastly, with $\epsilon=\frac{\gamma^2}{128}$, observe that $\frac{4\epsilon}{\alpha-\eta}\leq\frac{\gamma}{4}$. 
	
	Let $D$ be the oracle for \cref{lem:oraclebound}. For a close by polynomial $\sum_{i=0}^db_ix^i\in\mathbb{R}[x]$ such that $t:=\|\mathbf{b}-\mathbf{a}\|\geq h_j$ and $\sum_{i=0}^{d}a_ix^i=\sum_{i=0}^{d}b_ix^i$, by \cref{lem:errbound}, \cref{lem:oraclebound}, and \cref{lem:lemma11} respectively, we have, 
	\begin{align*}
          K_r(\mathbf{b})
          &\geq K_t(\mathbf{a})+K_{r-t,t}(x|\mathbf{a})-O(\log r)\\
          &\geq K_t^D(\mathbf{a})+K_{r-t,t}^D(x|\mathbf{a})-O(\log r)\\
		&\geq K_t^D(\mathbf{a})+K_{r-t,t}(x|\mathbf{a})-O(\log r)\\
		&\geq K_t^D(\mathbf{a})+s(r-t)-O(\log r).
	\end{align*}
	To ensure the conditions of \cref{lem:oraclebound}, we have 2 cases, $K_t^D(\mathbf{a})=\eta r$, and $K_t^D(\mathbf{a})=K_t(\mathbf{a})$. 
	
	For $K_t^D(\mathbf{a})=\eta r$, we have, 
	\begin{align*}
		K_r(\mathbf{b})&\geq\eta r+s(r-t)-O(\log r)\\
		&\geq\eta r+s(r-t)-\epsilon r\\
		&=(\eta-\epsilon)r+s(r-t)\\
		&\geq(\eta-\epsilon)r+(\alpha-\eta)(r-t).
	\end{align*}
	
	For $K_t^D(\mathbf{a})=K_t(\mathbf{a})$, we have, 
	\begin{align*}
		K_r(\mathbf{b})&\geq K_t(\mathbf{a})+s(r-t)-O(\log r)\\
		&\geq\rho t-o(t)+s(r-t)-O(\log r),&\textit{ by definition of $\dim$}\\
		&\geq\rho h_j+\rho(t-h_j)+s(r-t)-o(r)\\
		&=\rho h_j+\rho(t-h_j)+s(r-h_j)-s(t-h_j)-o(r)\\
		&=\alpha r+(\rho-s)(t-h_j)-o(r),&\textit{ by definition of $\alpha$}\\
		&=\eta r+(\alpha-\eta)r+(\rho-s)(t-h_j)-o(r)\\
		&\geq\eta r+(\alpha-\eta)(r-t)+(\rho-s)(t-h_j)-o(r)\\
		&\geq\eta r+(\alpha-\eta)(r-t)-o(r),&\textit{since }\textit{ $\rho\geq s$}\\
		&\geq\eta r-\epsilon r+(\alpha-\eta)(r-t)\\
		&=(\eta-\epsilon)r+(\alpha-\eta)(r-t).
	\end{align*}
	Hence, for oracle $D$, with $\delta=\alpha-\eta$, by \cref{lem:errorestimate}, we have, 
	\begin{align*}
		K_r\left(x,\sum_{i=0}^{d}a_ix^i\right)&\geq K_r(\mathbf{a},x)-\frac{4\epsilon}{\delta}(d+1)r-K(\epsilon)-K(\eta)-O_\mathbf{a}(\log r)\\
		&\geq K_r^D(\mathbf{a},x)-\frac{4\epsilon}{\delta}(d+1)r-K(\epsilon)-K(\eta)-O_\mathbf{a}(\log r)\\
		&\geq K_r^D(\mathbf{a},x)-K_{h_j,r}\left(\mathbf{a},x\middle|x,\sum_{i=0}^{d}a_ix^i\right)-\frac{4\epsilon}{\delta}(d+1)r-K(\epsilon,\eta)-O(\log r)\\
		&\geq K_r^D(\mathbf{a},x)-K_{h_j,r}\left(\mathbf{a},x\middle|x,\sum_{i=0}^{d}a_ix^i\right)-\frac{4\epsilon}{\alpha-\eta}(d+1)r-K(\epsilon,\eta)-O(\log r)\\
		&\geq K_r^D(\mathbf{a},x)-K_{h_j,r}\left(\mathbf{a},x\middle|x,\sum_{i=0}^{d}a_ix^i\right)-\frac{\gamma}{4}(d+1)r-\frac{\gamma}{8}(d+1)^2r\\
		&\geq K_r^D(\mathbf{a},x)-K_{h_j,r}\left(\mathbf{a},x\middle|x,\sum_{i=0}^{d}a_ix^i\right)-\frac{3\gamma}{8}(d+1)^2r.	\end{align*}
	By symmetry of information, we have, 
	\begin{align*}
		K_r^D(\mathbf{a},x)&=K_r^D(\mathbf{a})+K_{r,r}^D(x|\mathbf{a})-O(\log r)\\
		&=K_r^D(\mathbf{a})+K_{r,r}(x|\mathbf{a})-O(\log r),&\textit{ by \cref{lem:oraclebound}.1}\\
		&\geq\eta r+K_{r,r}(x|\mathbf{a})-O(\log r),&\textit{ by \cref{lem:oraclebound}.2}\\
		&\geq\eta r+sh_j+r-h_j-O(\log r)\\
		&\geq\alpha r-\frac{\gamma}{4}(d+1)^2r+sh_j+r-h_j-O(\log r)\\
		&\geq s(r-h_j)+\rho h_j-\frac{\gamma}{4}(d+1)^2r+sh_j+r-h_j-O(\log r),&\textit{ by definition of $\alpha$}\\
		&\geq(1+s)r-(1-\rho)h_j-\frac{\gamma}{2}(d+1)^2r.
	\end{align*}
	By \cref{lem:lessleft}, we have, 
	\begin{align*}
		K_{h_j,r}\left(\mathbf{a},x\middle|x,\sum_{i=0}^{d}a_ix^i\right)&\leq K_{h_j,h_j}\left(\mathbf{a},x\middle|x,\sum_{i=0}^{d}a_ix^i\right),&\textit{$\textit{since } r>h_j$}\\
		&=K_{h_j}(\mathbf{a},x)-K_{h_j}\left(x,\sum_{i=0}^{d}a_ix^i\right),&\textit{ by \cref{lem:soi}}\\
		&=K_{h_j}(\mathbf{a})+K_{h_j}(x|\mathbf{a})-K_{h_j}\left(x,\sum_{i=0}^{d}a_ix^i\right),&\textit{ by \cref{lem:soi}}\\
		&=K_{h_j}(\mathbf{a})+sh_j-K_{h_j}\left(x,\sum_{i=0}^{d}a_ix^i\right),&\textit{ by \cref{lem:lemma11}}\\
		&\leq K_{h_j}(\mathbf{a})+sh_j-\left(s+\rho-\frac{\gamma}{16}(d+1)^2\right)h_j,&\textit{ by \cref{lem:lessleft}}\\
		&\leq\rho h_j+\frac{h_j}{j}+sh_j-\left(s+\rho-\frac{\gamma}{16}(d+1)^2\right)h_j,&\textit{ definition of $h_j$}\\
		&\leq\frac{h_j}{j}+\frac{\gamma}{16}(d+1)^2r,&\textit{since }\textit{$h_j<r$}\\
		&\leq\frac{\gamma}{8}(d+1)^2r.&\textit{ for $j$ large}
	\end{align*}
	Hence, by combining the above results, we get,
	\begin{align*}
		K_r\left(x,\sum_{i=0}^{d}a_ix^i\right)&\geq K_r^D\left(\mathbf{a},x\right)-K_{h_j,r}\left(\mathbf{a},x\middle|x,\sum_{i=0}^{d}a_ix^i\right)-\frac{3\gamma}{8}(d+1)^2r\\
		&\geq(1+s)r-(1-\rho)h_j-\frac{\gamma}{2}(d+1)^2r-\frac{\gamma}{8}(d+1)^2r-\frac{3\gamma}{8}(d+1)^2r\\
		&=(1+s)r-(1-\rho)h_j-\gamma(d+1)^2r.
	\end{align*}
	Now, 
	\begin{align*}
		(1+s)r-(1-\rho)h_j-(s+\rho)r&=(1-\rho)r-(1-\rho)h_j\\
		&=(r-h_j)(1-\rho)\\
		&\geq0.&\textit { since $r\geq h_j$ and $\rho\leq1$}
	\end{align*}
	Therefore, 
	\begin{align*}
		K_r\left(x,\sum_{i=0}^{d}a_ix^i\right)&\geq\left(s+\rho-\gamma(d+1)^2\right)r.\qedhere
	\end{align*}
\end{proof}
\begin{proof}[Proof of \cref{lem:lessup}]
	\begin{align*}
		K_{h_j}\left(x,\sum_{i=0}^{d}a_ix^i\right)&\leq K_{h_j}(x,\mathbf{a})\\
		&=K_{h_j}(\mathbf{a})+K_{h_j}(x|\mathbf{a})\\
		&\leq\rho h_j+sh_j\qedhere
	\end{align*}
\end{proof}

\subsection{Error bounds for polynomial approximation}

\begin{proof}
  [Proof of \cref{lem:degree_d_approx}]
  If $a=b$, then the above bound holds trivially. Without loss of
  generality, assume that $|a|\ge|b|$. Then, we have
\begin{align*}
\left|a^k - b^k\right|
&< |a-b| \left|\left(\sum_{i=0}^{k-1}a^ib^{k-i-1}\right)\right|\\
&\le |a-b| \left(\sum_{i=0}^{k-1}\left|a^ib^{k-i-1}\right|\right)\\
&\le |a-b| \sum_{i=0}^{k-1} \left|a^{k-1}\right|\\
&\le 2^{-r} \sum_{i=0}^{k-1} \left|a^{k-1}\right|\\
&\le 2^{-r}k\times\max\{|a|,|a|^k\},
\end{align*}
where the third inequality follows since $|a|>|b|$. To establish the
third inequality, For the final inequality, it suffices to note
that if $|a|\le 1$, then $|a|\ge |a|^\ell$ for all $\ell \ge 1$,
otherwise, $|a|^\ell \le |a|^k$ for all $1 < \ell \le k$.
\end{proof}

The following lemma is used in upper bounding the number of non-root
points at which a polynomial assumes a small value. 

\begin{lemma}
  \label{lem:poly_bound}
  Let
  $p(x) = q(x) (x-\alpha)^m \prod_{\beta \in R(p)\setminus\{\alpha\}}
  (x-\beta)$ be a degree $d$ polynomial, where $R(p)$ is the set of
  real roots of $p$, $1 \le m \le d$and $q(x)$ is an irreducible
  polynomial over $\R$. Then, there is a rational constant $c_{p,q}$
  such that for all sufficiently large $r \in \mathbb{N}$, if
  $|y-\alpha|>2^{-r}$, then $|p(y)|>c_{p,q} 2^{-dr}$.
\end{lemma}
\begin{proof}
  We first consider the term
  \begin{align*}
    p'(x) = (x-\alpha)^m \prod_{\beta \in R(p)\setminus\{\alpha\}}
    (x-\beta). 
  \end{align*}
  Let
  $\Delta = \min \{\beta_i - \beta_j \mid \beta_i \ne \beta_j,
  \beta_i, \beta_j \in R(p)\}$ be the minimum separation between
  distinct roots of $p$. Note that $q(x)$ does not have any real root,
  so the real roots of $p$ and $p'$ are the same. Assume that the
  degree of $p'$ is $k$, where $1 \le m \le k \le d$.

  Suppose $r$ is sufficiently large that $2^{-r}<\frac{\Delta}{2}$.
  Let $|x - \beta| > \frac{\Delta}{2}$ for all roots $\beta$ distinct
  from $\alpha$. We have
  \begin{align*}
    |p'(x)|
    &= |x-\alpha|^m \prod_{\beta \in R(p)\setminus\{\alpha\}} |x -
      \beta|\\
    &\ge |x-\alpha|^m \left(\frac{\Delta}{2}\right)^{k-m}\\
    &\ge |x-\alpha|^m \min\left\{1,\frac{\Delta}{2}\right\}\\
    &\ge 2^{-mr} \min\left\{1,\frac{\Delta}{2}\right\}\\
    &\ge 2^{-dr} \min\left\{1,\frac{\Delta}{2}\right\}.
  \end{align*}
  Now, the second term can be bounded by the coefficients of $p'$
  using well-known bounds. This bound can then be lower bounded by a
  positive rational. Hence there is a rational constant $c_{p'}$,
  which depends only on $p'$, and not on $r$ and $x$, such that
  \begin{align*}
    |p'(x)| \ge  c_{p'} 2^{-dr}.
  \end{align*}

  Now, consider the irreducible polynomial $q(x)$. Since it has no
  real roots, there is a minimum positive value or a maximum negative
  value that it attains. Let $c_q$ be a positive rational which is
  less than $\min_{x \in \R} |q(x)|$.

  Letting $c_p=c_{p'}c_q$, we get
  \begin{align*}
    |p(x)| \ge c_{p'} c_{q} 2^{-dr} = c_{p} 2^{-dr}.
  \end{align*}
\end{proof}
\end{document}